\documentclass[11pt]{article}

\usepackage{amsmath}
\usepackage{amssymb}
\usepackage{amsfonts}
\usepackage{amsthm}
\usepackage{array}
\usepackage{bm}
\usepackage[shortlabels]{enumitem}
\usepackage{stmaryrd}

\usepackage{amsxtra}
\usepackage{array}
\usepackage{mathrsfs}
\usepackage{slashed}
\usepackage{xcolor}
\usepackage{tikz-cd}
\usepackage{tkz-euclide} 
\usepackage{pgfplots}

\newcolumntype{C}[1]{>{\centering\hspace{0pt}}p{#1}}
\usepackage[margin=1in]{geometry}

\newcommand{\GL}{\mathrm{GL}}

\newcommand{\SO}{\mathrm{SO}}

\newcommand{\U}{\mathrm{U}}
\newcommand{\SU}{\mathrm{SU}}

\newcommand{\G}{\mathrm{G}}

\newcommand{\Ker}{\mathrm{Ker}}

\newcommand{\Sym}{\mathrm{Sym}}

\newcommand{\Z}{\mathbb{Z}}

\newcommand{\R}{\mathbb{R}}
\newcommand{\C}{\mathbb{C}}

\newcommand{\RP}{\mathbb{RP}}

\newcommand{\vol}{\mathrm{vol}}

\newcommand{\Hc}{\mathcal{H}}

\newcommand{\uC}{\underline{\mathbb{C}}}

\newcommand{\sff}{\mathrm{I\!I}}

\newtheorem{thm}{Theorem}[section]
\newtheorem{prop}[thm]{Proposition}
\newtheorem{lem}[thm]{Lemma}
\newtheorem{cor}[thm]{Corollary}

\theoremstyle{definition}
\newtheorem{defn}[thm]{Definition}

\newtheorem{rmk}[thm]{Remark}

\newtheorem{question}[thm]{Question}

\usepackage[nottoc]{tocbibind}
\numberwithin{equation}{section}

\usepackage{hyperref}
\hypersetup{
	colorlinks,
	linkcolor={red!50!black},
	citecolor={blue!50!black},
	urlcolor={blue!80!black}
}

\title{The Morse index of constant curvature $2$-spheres}
\author{Gavin Ball and Jesse Madnick}
\date{June 2026}

\newcommand{\Addresses}
{{  \bigskip
\noindent	\textsc{University of Missouri} \par\nopagebreak
\noindent	\textsc{Columbia, MO, United States} \par\nopagebreak
\noindent	\texttt{gavin.ball@missouri.edu} \\

\medskip
\noindent	\textsc{Seton Hall University} \par\nopagebreak
\noindent	\textsc{South Orange, NJ, United States} \par\nopagebreak
\noindent	\texttt{jesse.ochs.madnick@gmail.com} \\
}}

\begin{document}

\maketitle

\begin{abstract} In the round $N$-sphere, we calculate the Morse index and nullity of all immersed minimal $2$-spheres having constant Gauss curvature.  We also obtain bounds on the stability index of the associative cone in $\R^7$ whose link is the Boruvka sphere in $S^6$.
\end{abstract}

\tableofcontents

\section{Introduction} \label{sec:Intro}

\indent \indent The study of $2$-dimensional minimal surfaces $\Sigma$ in round spheres $S^N$, $N \geq 3$, is a classical subject in differential geometry, yet many outstanding questions remain.  Several of these concern the Morse index, $\mathrm{Ind}(\Sigma)$, a non-negative integer invariant that counts the number of vector fields along $\Sigma$ that are area-decreasing to second-order.  Unfortunately, while it is the subject of various deep and influential conjectures, the Morse index is very difficult to compute, even for relatively simple minimal surfaces. \\
\indent As of this writing, the majority of results in the literature concern the codimension one case of $N = 3$.  By contrast, our understanding in the high codimension regime appears to be rather limited at present.  For any $N \geq 3$, we are aware of the works of Simons \cite{Simons68}, Ejiri \cite{Ejiri83}, Karpukhin \cite{Karpukhin21}, and Kusner--Wang \cite{KusWang24}.  In specific dimensions, there are further results of Micallef--Wolfson \cite{MicaWolf93} and Montiel--Urbano \cite{MontUrb97} for $N = 4$, of Itoh \cite{Itoh00} and Urbano \cite{Urbano03} for $N = 5$, and of the second author \cite{Madnick22} for $N = 6$, to name a few examples. \\
\indent In this paper, we will be concerned with minimal 2-spheres in $S^N$, for all $N \geq 2$, a subject with a long and rich history.  Almost a century ago, for each $n \geq 1$, Bor{\r{u}}vka \cite{Boruvka33} discovered an immersed minimal 2-sphere $\Sigma_n$ in $S^{2n}$ that is linearly full (i.e., is contained in no totally geodesic sub-sphere) and has constant Gauss curvature $K = \frac{2}{n(n+1)}$.  Bor{\r{u}}vka’s examples turn out to be fundamental.  Indeed, Calabi \cite{Calabi67} proved that every minimal $2$-sphere in $S^N$ with constant Gauss curvature is congruent to some Bor{\r{u}}vka sphere in a totally geodesic $S^{2n} \subset S^N$.  Shortly thereafter, Wallach \cite{Wallach70} established an even stronger statement: every connected piece of a minimal surface in $S^N$ with constant $K > 0$ is a subset of a Bor{\r{u}}vka sphere.  Later, Ejiri \cite{Ejiri86} showed that the cone $\mathrm{C}(\Sigma_n) \subset \R^{2n+1}$ with link $\Sigma_n$ is stable. \\
\indent Our first result, proved in $\S$\ref{sec:LinFull}, is an explicit formula for the Morse index and nullity of the Bor{\r{u}}vka spheres.

\begin{thm} \label{thm:BoruvkaFull} The Morse index and nullity of the Bor{\r{u}}vka sphere $\Sigma_n$ in $S^{2n}(1)$ are:
\begin{align*}
\mathrm{Ind}(\Sigma_n) & = n(n-1)(2n+1), & \mathrm{Nul}(\Sigma_n) & = (2n+3)(2n-2).
\end{align*}
In particular, every Jacobi field of $\Sigma_n$ arises from a twistor deformation.  (Here and in the sequel, we always mean the index and nullity of the immersion.)
\end{thm}

\indent Our proof of Theorem \ref{thm:BoruvkaFull} draws on results from the theory of tridiagonal matrices.  More specifically, we exploit the $\SO(3)$-invariance of $\Sigma_n$ to reduce the computation to an analysis of the spectrum of a finite collection of tridiagonal matrices whose entries depend polynomially on two parameters.  In $\S$\ref{sec:HighCodim}, we strengthen this result.  Indeed, using Theorem \ref{thm:BoruvkaFull} together with Calabi’s result mentioned above, a short argument yields the index and nullity of all minimal 2-spheres in $S^N$ with constant Gauss curvature.

\begin{cor} \label{cor:HighCodim} Let $\Sigma$ be an immersed minimal $2$-sphere in $S^{2n+a}(1)$ for $n \geq 1$ and $a \geq 0$.  If $\Sigma$ has constant Gauss curvature, so that $\Sigma$ is linearly full in some totally geodesic $2n$-sphere, then its Morse index and nullity are
\begin{align*}
\mathrm{Ind}(\Sigma) & = n(n-1)(2n+1) + an^2, & \mathrm{Nul}(\Sigma) & = (2n+3)(2n-2) + a(2n+1).
\end{align*}
\end{cor}

Heuristically, we expect that ``simpler" minimal surfaces ought to have low Morse indices.  Given that the Bor{\r{u}}vka sphere has the simplest possible Gauss curvature, it is natural to wonder whether it also has a relatively low index.  This raises the following:

\begin{question} \label{question} If $N$ is an immersed minimal $2$-sphere in $S^{2n}(1)$ that is linearly full, then does it follow that
$$\mathrm{Ind}(N) \geq n(n-1)(2n+1)?$$
Moreover, does equality hold if and only if $N$ is twistor-equivalent to the Bor{\r{u}}vka sphere $\Sigma_n$?
\end{question}

For $n = 2$, a theorem of Montiel--Urbano \cite{MontUrb97} provides an affirmative answer to Question \ref{question}.  To our knowledge, the best known lower bound to date is due to Karpukhin \cite{Karpukhin21}, who proves that
$$\mathrm{Ind}(N) \geq 2(n-1)\left( 2d + 2 - \left\lfloor \sqrt{8d + 1} \right\rfloor_{\mathrm{odd}}  \right)\!,$$
where $d = \frac{1}{4\pi}|N|$ is the degree of the immersion, and $\lfloor x\rfloor_{\mathrm{odd}}$ denotes the largest odd integer not exceeding $x$. \\

\indent In addition to its significance as a minimal surface, the Bor{\r{u}}vka sphere $\Sigma_3 \subset S^6$ is also important from the perspective of $\G_2$ geometry, as we now explain.  Let us equip the $6$-sphere with its standard ($\G_2$-invariant, non-integrable) almost-complex structure $J$ arising from the octonion algebra.  A surface $\Sigma \subset S^6$ is called a $J$-holomorphic curve if $J(T\Sigma) = T\Sigma$.  The Bor{\r{u}}vka sphere $\Sigma_3$ is the simplest linearly full $J$-holomorphic curve in $S^6$. \\
\indent It turns out that a surface $\Sigma \subset S^6$ is a $J$-holomorphic curve if and only if its cone $C = \mathrm{C}(\Sigma) = \{r \sigma \in \R^7 \colon r > 0, \sigma \in \Sigma\}$ is an associative $3$-fold in $\R^7$.  By ``associative $3$-fold,” we mean a $3$-dimensional submanifold in $\R^7$ that is calibrated by the $3$-form $\varphi \in \Omega^3(\R^7)$ defined by $\varphi(u,v,w) = \langle u \times v, w\rangle$, where $\times$ is the vector cross product in $\R^7$.  The upshot is that $J$-holomorphic curves in $S^6$ serve as models for conical singularities of associative $3$-folds in $\R^7$. \\
\indent The deformation theory of conically singular associative $3$-folds has been studied systematically by Bera \cite{Bera23}.  For an associative $3$-fold with a conical singularity modeled on $C$, Bera proves \cite[Theorem 1.12]{Bera23} that the Fredholm index of its deformation operator is equal to the negative of the ``stability index" $\text{s-ind}(C)$, an integer that depends only on the cone $C$.  Like the Morse index, the stability index is difficult to compute in general. \\
\indent The upper and lower stability indices, integers denoted $\text{s-ind}_+(C)$ and $\text{s-ind}_-(C)$, provide upper and lower bounds on the stability index, respectively.  These, too, are challenging to calculate explicitly, as they require detailed knowledge of the spectrum of a Dirac operator defined on the normal bundle of the link $\Sigma$ of $C$. \\
\indent In $\S$\ref{sec:StabIndex}, we obtain the spectrum of the relevant Dirac operator on the normal bundle of the Bor{\r{u}}vka sphere $\Sigma_3$, and thereby deduce the following result:

\begin{thm} \label{thm:StabilityResult} Let $C = \mathrm{C}(\Sigma_3) \subset \R^7$ be the (associative) cone over the Bor{\r{u}}vka sphere in $S^6$.  Then its upper and lower stability indices are $\mathrm{s}\text{-}\mathrm{ind}_+(C) = 41$ and $\mathrm{s}\text{-}\mathrm{ind}_-(C) = 30$.  In particular, $C$ is not rigid as an associative cone.
\end{thm}

\noindent \textbf{Acknowledgements:} We thank Gorapada Bera and Mikhail Karpukhin for helpful conversations.

\section{The linearly full case in $S^{2n}$} \label{sec:LinFull}

\subsection{The Bor{\r{u}}vka sphere}

\indent \indent Classically, the Bor{\r{u}}vka sphere $\Sigma_n \subset S^{2n}$ is defined as follows.  Let $\{p_1, \ldots, p_{2n+1}\}$ be an orthogonal basis of the set of harmonic, homogeneous $n$th degree polynomials in $3$ variables, normalized so that $(p_1)^2 + \cdots + (p_{2n+1})^2 = 1$.  Then
$$\Sigma_n = \left\{ \left( p_1(x,y,z), \ldots, p_{2n+1}(x,y,z) \right) \in \R^{2n+1} \colon x^2 + y^2 + z^2 = 1 \right\}.$$
For example, $\Sigma_2 \subset S^4$ is the Veronese surface
$$\Sigma_2 = \textstyle \left\{ \frac{1}{\sqrt{3}}\!\left( xy, xz, yz, \frac{1}{2}(x^2 - y^2), \frac{1}{2\sqrt{3}}(x^2 + y^2 - 2z^2)\right) \in \R^5 \colon x^2 +y^2 + z^2 = 1 \right\}\!.$$
Topologically, each $\Sigma_{2k+1}$ is diffeomorphic to $S^2$, whereas each $\Sigma_{2k}$ is diffeomorphic to $\RP^2$.  It is well-known that $\Sigma_n$ is $\SO(3)$-invariant and has constant curvature $K = \frac{2}{n(n+1)}$.   \\
\indent In this work, we adopt a different point of view, regarding the Bor{\r{u}}vka sphere as the unique $2$-dimensional $\SO(3)$-orbit in $S^{2n} \subset \R^{2n+1}$ up to congruence, where $\SO(3)$ acts irreducibly on $\R^{2n+1}$.  To be more concrete, let us recall that every irreducible real $\SO(3)$-representation is isomorphic to
$$\mathcal{H}_k = \{\text{harmonic homogeneous }k\text{th degree polynomials in 3 variables}\} \cong \R^{2k+1}, \ \ k \geq 0.$$
Thus, viewing the ambient euclidean space $\R^{2n+1} = \mathcal{H}_n$ as a space of polynomials, and letting $\pi \colon \Sym^n(\R^3) \to \mathcal{H}_n$ denote a suitable scaling of the projection of a polynomial onto its harmonic part, the Bor{\r{u}}vka sphere may be regarded as $\Sigma_n = \SO(3) \cdot (\pi(x^n))$. \\
\indent We now explore some aspects of the geometry of $\Sigma_n$ by means of the moving frame.  For this, we let $\mathbf{g} \colon \SO(3) \to \mathrm{M}_{2n+1}(\R)$ denote the embedding given by the $\SO(3)$-action on $\Hc_n = \R^{2n+1}$, and write $\mathbf{g} = \begin{bmatrix} \mathbf{x} & \mathbf{e}_1 & \cdots & \mathbf{e}_{2n} \end{bmatrix}$.  Let $\mu \in \Omega^1(\SO(3); \mathfrak{so}(3))$ denote the Maurer-Cartan form of $\SO(3)$, so that $\mu = \mathbf{g}^{-1}d\mathbf{g}$, and hence
\begin{align} \label{eq:StrEq0}
d\!\begin{bmatrix} \mathbf{x} & \mathbf{e}_1 & \cdots & \mathbf{e}_{2n} \end{bmatrix} =  \begin{bmatrix} \mathbf{x} & \mathbf{e}_1 & \cdots & \mathbf{e}_{2n} \end{bmatrix} \mu.
\end{align}
Geometrically, we regard elements $\begin{bmatrix} \mathbf{x} & \mathbf{e}_1 & \cdots & \mathbf{e}_{2n} \end{bmatrix} \in \SO(3)$ as adapted $\SO(2)$-frames along $\Sigma_n$. \\
\indent To be more explicit, we shall use the embedding $\mathfrak{so}(3) \hookrightarrow \mathfrak{so}(2n+1)$ to write $\mu$ as a matrix-valued $1$-form.  Indeed, following \cite{Bryant85}, there exist $1$-forms $\omega_1, \omega_2, \rho \in \Omega^1(\SO(3))$ such that
\begin{equation} \label{eq:so3irrmc}
   \mu = \left[ \begin{array}{ c  c | c  c  c  c }
        0 & -\Omega^t & 0 & 0 & \cdots & 0 \\
        \Omega & \Lambda & - c_2 \Psi^t & 0 & \cdots & 0 \\ \hline
        0 & c_2 \Psi & 2 \Lambda & - c_3 \Psi^t & \ddots & 0 \\
        0 & 0 & c_3 \Psi & 3 \Lambda & \ddots & \\
        \vdots & \vdots & 0 & \ddots & \ddots & \\
         & & \vdots & \ddots &(n-1) \Lambda & -c_k \Psi^t \\
         0 & 0 & 0 & & c_k \Psi & n \Lambda
    \end{array}\right]\!,
\end{equation}
where
\begin{equation*}
    \begin{aligned}
        \Omega = \begin{bmatrix}
            \omega_1 \\
            \omega_2
        \end{bmatrix}, \:\:\:\: \Psi = \begin{bmatrix}
            \omega_1 & - \omega_2 \\
            \omega_2 & \omega_1
        \end{bmatrix}, \:\:\:\: \Lambda = \begin{bmatrix}
            0 & \rho \\
            -\rho & 0
        \end{bmatrix},
    \end{aligned}
\end{equation*}
and the constants $c_r$  are
\begin{equation*}
    c_r = \sqrt{\frac{1}{2} \left( 1 - \frac{r(r-1)}{n(n+1)}  \right)}, \ \ \ 2 \leq r \leq n.
\end{equation*}
\indent Now, equations (\ref{eq:StrEq0}) and (\ref{eq:so3irrmc}) imply that $d\mathbf{x} = \mathbf{e}_1\omega_1 + \mathbf{e}_2\omega_2$, from which we may interpret $\mathbf{e}_1, \mathbf{e}_2$ as tangent vectors to $\Sigma_n$, and $\mathbf{e}_3, \ldots, \mathbf{e}_{2n}$ as normal vectors.  Consequently, we see that $\Lambda$ (or equivalently, $\rho$) is the connection form for the tangential connection on $T\Sigma_n$, while the lower right block of (\ref{eq:so3irrmc}) is the connection form for the normal connection on $N\Sigma_n$.  Moreover, from the lower left block of (\ref{eq:so3irrmc}), we see that the second fundamental form coefficients are given by
\begin{equation} \label{eq:SFF}
    \sff^3 = c_2 \begin{bmatrix}
        1 & 0 \\
        0 & - 1
    \end{bmatrix}, \:\:\: \sff^4 = c_2 \begin{bmatrix}
        0 & 1 \\
        1 & 0
    \end{bmatrix}, \:\:\: \sff^a = 0 \:\:\: \text{for} \:\: a \geq 5.
\end{equation}
Finally, note that the Maurer-Cartan equation $d\mu = -\mu \wedge \mu$ yields the standard structure equations
\begin{align} \label{eq:StrEq1}
d\omega_1 & = -\rho \wedge \omega_2  \notag \\
d\omega_2 & = \rho \wedge \omega_1 \\
d\rho & = K\,\omega_1 \wedge \omega_2, \notag
\end{align}
with $K = \frac{2}{n(n+1)}$.  In particular, this verifies that $\Sigma_n$ has constant curvature $K$. \\

\indent In the sequel, we shall make frequent use of certain associated complex line bundles.  That is, for $k \in \Z$, we let $\C_k$ denote the complex $\SO(2)$-module of highest weight $k$, and let $\rho_k \colon \SO(2) \to \SO(\C_k)$ denote the corresponding representation (i.e., $\rho_k(e^{i\theta}) = e^{ki\theta}$).  We then let
$$\underline{\C}_k := \SO(3) \times_{\rho_k} \C_k \to S^2$$
denote the associated complex line bundle.    From (\ref{eq:so3irrmc}), we observe that the tangent and normal bundles of $\Sigma_n$ are isomorphic as real $\SO(2)$-vector bundles to
\begin{align}
   T \Sigma_n & = \uC_1, &    N \Sigma_n & = \uC_2 \oplus \uC_3 \oplus \cdots \oplus \uC_n. \label{eq:normbundiso}
\end{align}
In particular, for $k \geq 1$, we see that $\uC_k \cong \Sym^k(T^{1,0}S^2)$ as real vector bundles, so that the normal bundle is  $N\Sigma_n \cong \bigoplus_{r = 2}^n \Sym^r(T^{1,0}S^2)$, which is intrinsic.  Indeed, we shall often write sections $V \in \Gamma(N\Sigma_n)$ as tuples of symmetric tensor fields, say $V = (A_2, A_3, \ldots, A_n)$ where each $A_r \in \Gamma(\uC_r)$.

\subsection{The Jacobi operator}

\indent \indent In general, for a $2$-dimensional minimal surface $\Sigma$ in a Riemannian manifold $M$ with second fundamental form $\mathrm{I\!I}$, we define operators $\mathscr{B}, \mathscr{R} \colon\Gamma(N\Sigma) \to \Gamma(N\Sigma)$ by
\begin{align*}
\mathscr{B}(\eta) & = \left\langle \mathrm{I\!I}(e_i, e_j), \eta \right\rangle \mathrm{I\!I}(e_i, e_j), & \mathscr{R}(\eta) = \left( \overline{R}(\eta, e_i)e_i \right)^\perp,
\end{align*}
where $\overline{R}$ is the curvature of $M$, and $(e_1, e_2)$ is a local orthonormal frame for $T\Sigma$.  The Jacobi operator  $\mathscr{J} \colon \Gamma(N\Sigma) \to \Gamma(N\Sigma)$ of the minimal surface is given by
\begin{equation*}
    \mathscr{J} = - \Delta^\perp - \mathscr{B} - \mathscr{R},
\end{equation*}
where $\Delta^\perp$ is the connection Laplacian associated to the normal connection on $N\Sigma$. \\
\indent In our situation, $M = S^{2n}$ is the round sphere of constant curvature $1$, and $\Sigma = \Sigma_n$ is the Bor{\r{u}}vka sphere.  In this case, it is well-known that $\mathscr{R} = 2 \, \mathrm{Id}_{N\Sigma_n}$.  Moreover, by (\ref{eq:SFF}), the operator $\mathscr{B}$ is given by
\begin{equation*}
        \mathscr{B} = 2 c_2^2 \, \mathrm{Id}_{\uC_2} \oplus 0_{\uC_3}  \oplus \cdots \oplus 0_{\uC_n}.
\end{equation*}
It remains to calculate $\Delta^\perp$.  The difficulty is that $\Delta^\perp$ (and hence also $\mathscr{J}$) does not respect the splitting $N\Sigma \cong \Sym^2(T^{1,0}S^2) \oplus \cdots \oplus \Sym^n(T^{1,0}S^2)$. \\
\indent As such, we shall introduce a second Laplace operator.  That is, using the isomorphism $N\Sigma \cong \Sym^2(T^{1,0}S^2) \oplus \cdots \oplus \Sym^n(T^{1,0}S^2)$ to view normal vectors as tuples of symmetric tensors, we may equip $N\Sigma$ with the direct sum connection $\nabla^{(2)} \oplus \cdots \oplus \nabla^{(n)}$, in which each $\nabla^{(r)}$ on $\Sym^r(T^{1,0}S^2)$ is induced from the Levi-Civita connection $\rho$ on $S^2$.  This direct sum connection on $N\Sigma$ induces a connection Laplacian $\Delta$ that is well-suited to the geometry of the Bor{\r{u}}vka sphere.  We seek a formula relating $\Delta^\perp$ to $\Delta$. \\

\indent Make a complex change of variables $\eta = \omega_1 + i \omega_2$, so that the structure equations (\ref{eq:StrEq1}) become
\begin{equation*}
    d \eta = i \rho \wedge \eta, \:\:\:\: d \rho = \tfrac{i}{2} K \eta \wedge \overline{\eta}.
\end{equation*}
Note that each section $\sigma \in \Gamma(\underline{\C}_r) = \Gamma( \SO(3) \times_{\rho_r} \C_r)$ can be written as $\sigma = A \eta^r$ for a unique $\SO(2)$-equivariant function $A \colon \SO(3) \to \C_r$.  In the sequel, we identify sections $\sigma$ with their corresponding equivariant functions $A$.  With this understood, we define operators
\begin{equation*}
\begin{aligned}
    \partial &: \Gamma( \uC_r ) \to \Gamma(\uC_{r+1}), \\
    \overline{\partial} &: \Gamma( \uC_r ) \to \Gamma(\uC_{r-1}),
\end{aligned}
\end{equation*}
by the equation
\begin{equation}\label{eq:dA}
    d A + i r A \rho = \left( \partial A \right) \eta + \left( \overline{\partial} A \right) \overline{\eta}.
\end{equation}

\begin{rmk} We may define a complex structure on $S^2$ by declaring a $\C$-valued $1$-form on $S^2$ to be of type $(1,0)$ if its pullback to $\SO(3)$ lies in the complex span of $\eta$.  The $(0,1)$-part of the Levi-Civita connection $\nabla = \nabla^{(r)}$ then endows the line bundle $\underline{\C}_r \to S^2$ with a holomorphic structure.  In fact, $\nabla^{0,1} \colon \Gamma(\underline{ \mathbb{C}}_r  ) \to  \Gamma( \underline{ \mathbb{C}}_r \otimes \Lambda^{0,1}S^2 ) = \Gamma( \underline{ \mathbb{C}}_r \otimes \underline{\C}_{-1} ) \cong \Gamma( \underline{\mathbb{C}}_{r-1})$ corresponds to $\overline{\partial}$.
\end{rmk}

\begin{prop}\label{prop:weitzenbock}
    On $\Gamma(\uC_r),$ there is a Weitzenbock formula
    \begin{equation*}
        \left( \partial \overline{\partial} - \overline{\partial} \partial \right) A = -\tfrac{1}{2} r K A.
    \end{equation*}
\end{prop}
\begin{proof}
    Differentiating (\ref{eq:dA}) and using $d^2 = 0$ and the structure equations we have
    \begin{equation*}
        \begin{aligned}
            - i r \left[ d(A) \right ] \wedge \rho + i r A \left( \tfrac{i}{2} K \eta \wedge \overline{\eta} \right) &= \left[ d \left( \partial A \right) \right] \wedge \eta + \left( \partial A \right) \left( i \rho \wedge \eta \right) + \left[ d \left( \overline{\partial} A \right) \right] \wedge \overline{\eta} + \left( \overline{\partial} A \right) \left( -i \rho \wedge \overline{\eta} \right) \\
           \implies &  \tfrac{1}{2} r K A \eta \wedge \overline{\eta} + \overline{\partial} \partial A \overline{\eta} \wedge \eta + \partial \overline{\partial} A \eta \wedge \overline{\eta} = 0 \\
           \implies & \left( \partial \overline{\partial} A - \overline{\partial} \partial A + \tfrac{1}{2} r K A \right) \eta \wedge \overline{\eta} = 0.
        \end{aligned}
    \end{equation*}
\end{proof}

\begin{prop}\label{prop:CrLaplacian}
    Let $\Delta$ denote the connection Laplacian associated to the connection on $\uC_r \cong \Sym^r(T^{1,0}S^2)$ induced by the Levi-Civita connection $\rho$. Then
\begin{enumerate}[(a)]
\item We have $\Delta = 2 \left( \partial \overline{\partial} + \overline{\partial} \partial \right)$.
\item For $A \in \Gamma(\uC_r)$, we have the following commutation relations
    \begin{equation*}
        \begin{aligned}
            \left[ \Delta, \partial \right] A &= \left( 2r + 1 \right) K \partial A, \\
            \left[ \Delta, \overline{\partial} \right] A &= -\left( 2 r - 1 \right) K \overline{\partial} A,
        \end{aligned}
    \end{equation*}
\end{enumerate}    

    \begin{proof} Part (a) follows from a direct calculation.  Part (b) follows from combining Proposition \ref{prop:weitzenbock} and part (a).
    \end{proof}
\end{prop}

\begin{prop}\label{prop:normlap}
    Let $\Delta^\perp$ denote the connection Laplacian associated to the normal connection $\nabla^\perp$ on $N \Sigma_n$, and let $\Delta$ be the connection Laplacian associated to the direct sum connection $\nabla^{(2)} \oplus \cdots \oplus \nabla^{(n)}$ on $N\Sigma_n = \uC_2 \oplus \cdots \oplus \uC_n$.  Then
    \begin{equation*}
        \Delta^{\perp} \begin{bmatrix}
            A_2 \\
            A_3 \\
            \vdots \\
            A_{n-1} \\
            A_{n}
        \end{bmatrix} = \Delta \begin{bmatrix}
            A_2 \\
            A_3 \\
            \vdots \\
            A_{n-1} \\
            A_{n}
        \end{bmatrix} + 4 \, \partial \begin{bmatrix}
            0 \\
            c_3 A_2 \\
            \vdots \\
            c_{n-1} A_{n-2} \\
            c_{n} A_{n-1}
        \end{bmatrix} - 4 \, \overline{\partial} \begin{bmatrix}
            c_3 A_3 \\
            c_4 A_4 \\
            \vdots \\
            c_{n-1} A_{n-1} \\
            0
        \end{bmatrix} - 2 \begin{bmatrix}
            c_3^2 A_2 \\
            (c_3^2 + c_4^2) A_3 \\
            \vdots \\
            (c_{n-1}^2+c_n^2) A_{n-1} \\
            c_n^2 A_n
        \end{bmatrix}.
    \end{equation*}
\end{prop}

\begin{proof}
    From (\ref{eq:so3irrmc}), and switching to complex notation, the connection form for the normal connection is
    \begin{equation*}
        \begin{bmatrix}
           2 i \rho & - c_3 \eta & 0 & & & & \\
            c_3 \overline{\eta} & 3 i \rho & - c_4 \eta & 0 & & & \\
            0 & c_4 \overline{\eta} & 4 i \rho  & -c_5 \eta & \ddots  & &  \\
             & \ddots & \ddots & \ddots & \ddots & & \\
             & & & & (n-1)i \rho & -c_n {\eta} \\
             & & & & c_n \overline{\eta} & n i \rho
        \end{bmatrix}.
    \end{equation*}
    It follows that we have equations
    \begin{equation*}
        \partial^{\perp} = \partial - U, \:\:\:\: \overline{\partial}^\perp = \overline{\partial} + L,
    \end{equation*}
    where $U$ and $L$ are given by
    \begin{equation*}
        U = \begin{bmatrix}
            0 & c_3 & 0 & & \\
            0 & 0 & c_4 & 0 & \\
             & \ddots & \ddots & \ddots & \\
             & & & & c_n \\
             & & & 0 & 0
        \end{bmatrix}, \:\:\:\: L = \begin{bmatrix}
            0 & 0 &  & & \\
            c_3 & 0 & 0 &  & \\
             & c_4 & \ddots & \ddots & \\
             & & \ddots & & \\
             & & & c_n & 0
        \end{bmatrix}.
    \end{equation*} 
We may now calculate
    \begin{align*}
        \Delta^\perp & = 2 \left( \partial^\perp \overline{\partial}^\perp + \overline{\partial}^\perp \partial^\perp \right) \\
        & = 2 \left( (\partial - U)(\overline{\partial} + L) + (\overline{\partial} + L)(\partial - U)  \right) \\
        & = \Delta - 4  \overline{\partial}U + 4  \partial L - 2 \, (UL + LU),
    \end{align*}
where we have used that $U\overline{\partial} = \overline{\partial} U$ and $\partial L = L \partial$.
\end{proof}
\noindent Therefore, the Jacobi operator of $\Sigma_n$ is
\begin{equation} \label{eq:Jacobi-operator}
    \mathscr{J} \begin{bmatrix}
            A_2 \\
            A_3 \\
            \vdots \\
            A_{n-1} \\
            A_{n}
        \end{bmatrix} = - \Delta \begin{bmatrix}
            A_2 \\
            A_3 \\
            \vdots \\
            A_{n-1} \\
            A_{n}
        \end{bmatrix} - 4 \, \partial \begin{bmatrix}
            0 \\
            c_3 A_2 \\
            \vdots \\
            c_{n-1} A_{n-2} \\
            c_{n} A_{n-1}
        \end{bmatrix} + 4 \, \overline{\partial} \begin{bmatrix}
            c_3 A_3 \\
            c_4 A_4 \\
            \vdots \\
            c_{n-1} A_{n-1} \\
            0
        \end{bmatrix} + 2 \begin{bmatrix}
            (-c_2^2 + c_3^2 - 1) A_2 \\
            (c_3^2 + c_4^2 - 1) A_3 \\
            \vdots \\
            (c_{n-1}^2+c_n^2 - 1) A_{n-1} \\
            (c_n^2 - 1) A_n
        \end{bmatrix}\!.
\end{equation}

\begin{rmk} Above, we worked with the complex line bundles $\uC_r \cong \Sym^r(T^{1,0}S^2) \to S^2$.  An alternative approach is to work with the real vector bundles $\Sym^r_0(TS^2) \to S^2$.  Indeed, the Levi-Civita connection on $S^2$ induces a covariant derivative operator $\nabla \colon \Gamma(\Sym^r_0(TS^2)) \to \Gamma( \Sym^r_0(TS^2) \otimes TS^2)$.  Moreover, there is a decomposition
$$\Sym^r_0(TS^2) \otimes TS^2 \cong \Sym^{r+1}_0(TS^2) \oplus \Sym^{r-1}_0(TS^2),$$
so $\nabla = \partial^+ \oplus \partial^-$ for a pair of first-order differential operators $\partial^{\pm} \colon \Gamma(\Sym^r_0(TS^2)) \to \Gamma(\Sym^{r \pm 1}_0(TS^2))$.
\end{rmk}

\subsection{The Jacobi spectrum on isotypical components}

\indent \indent Our next step is to decompose the space of $L^2$ normal vector fields $L^2(\Sigma_n, N\Sigma_n)$ into a direct sum of isotypical components, and then to express the Jacobi operator on each component.  In general, for a compact homogeneous space $\G / \mathrm{H}$ with a homogeneous vector bundle $E \to \G / \mathrm{H}$ modeled on a $\mathrm{H}$-representation $E_0$, the Frobenius Reciprocity Theorem provides a $\G$-invariant decomposition
\begin{equation*}
    L^2(\G / \mathrm{H}, E) = \bigoplus_{\rho \in \widehat{\G}} V_{\rho} \otimes \mathrm{Hom}_{\mathrm{H}} (V_{\rho}, E_0),
\end{equation*}
where $\widehat{\G}$ is the set of equivalence classes of finite-dimensional, irreducible complex $\G$-modules.  For our homogeneous vector bundle $\uC_k \to S^2$ under consideration, we have $\G = \SO(3),$ $\mathrm{H} = \SO(2),$ $E_0 = \C_k$, and
\begin{equation*}
    \mathrm{Hom}_{\SO(2)} ( \Hc_\ell, \C_k) = \begin{cases}
        \left\lbrace 0 \right\rbrace & \ell < k \\
        \C & \ell \geq k,
    \end{cases}
\end{equation*}
so that Frobenius Reciprocity gives the $\SO(3)$-invariant decomposition
\begin{equation}\label{eq:frobrecipCk}
    L^2(S^2, \uC_k) = \bigoplus_{\ell = k}^\infty \Hc_\ell^\C
\end{equation}
in which $\Hc_\ell^\C = \Hc_\ell \otimes_{\R} \C$. Therefore, from (\ref{eq:normbundiso}) and (\ref{eq:frobrecipCk}), there is an $\SO(3)$-invariant decomposition
\begin{equation} \label{eq:Isotypical}
\begin{aligned}
    L^2(\Sigma_n, N \Sigma_n) & \cong L^2(S^2, \uC_2) \oplus L^2(S^2, \uC_3) \oplus \cdots \oplus L^2(S^2, \uC_n) \\
    & = \left( \Hc_2^\C \oplus \Hc_3^\C \oplus \Hc_4^\C \oplus \cdots \right) \oplus \left( \Hc_3^\C \oplus \Hc_4^\C \oplus \Hc_5^\C \oplus \cdots \right) \oplus \cdots \\
    & \ \ \ \ \ \cdots \oplus \left( \Hc_n^\C \oplus \Hc_{n+1}^\C \oplus \Hc_{n+2}^\C \oplus \cdots \right) \\
    & = \bigoplus_{\ell = 2}^\infty \left(\Hc_\ell^\C\right) ^{\oplus\, m(\ell)},
 \end{aligned}
\end{equation}
where the multiplicity $m(\ell)$ is given by
\begin{equation*}
    m(\ell) = \begin{cases}
        \ell - 1 & 2 \leq \ell \leq n - 1 \\
        n - 1 & \ell \geq n.
    \end{cases}
\end{equation*}
Since the immersion $\Sigma_n \to S^{2n}$ is $\SO(3)$-equivariant, the Jacobi operator $\mathscr{J}$ preserves each isotypical component $I_\ell := \left(\Hc_\ell^\C\right) ^{\oplus\, m(\ell)}$.  Note that $\dim(I_\ell) = m(\ell) d(\ell)$, where $d(\ell) := \dim(\Hc_\ell^\C) = 2(2\ell + 1)$.  The following diagram summarizes the situation:
$$\begin{tikzcd}
{L^2(N\Sigma_n)} \arrow[r, "=", phantom]       & L^2(\underline{\mathbb{C}}_2) \arrow[r, "\oplus", phantom] \arrow[r, "\partial", bend left=35]  & L^2(\underline{\mathbb{C}}_3) \arrow[r, "\oplus", phantom] \arrow[r, "\partial", bend left=35] \arrow[l, "\overline{\partial}", bend left=35] & L^2(\underline{\mathbb{C}}_4) \arrow[r, "\oplus", phantom] \arrow[l, "\overline{\partial}", bend left=35] & \cdots \arrow[r, "\oplus", phantom] & L^2(\underline{\mathbb{C}}_n) \\
I_2 \arrow[r, "=", phantom]                              & \mathcal{H}_2^{\mathbb{C}}                                                                     &                                                                                                &                                                            &                                     &                               \\
I_3 \arrow[r, "=", phantom]                              & \mathcal{H}_3^{\mathbb{C}} \arrow[r, "\oplus", phantom]                                        & \mathcal{H}_3^{\mathbb{C}}                                                                     &                                                            &                                     &                               \\
I_4 \arrow[r, "=", phantom] \arrow[d, "\vdots", phantom] & \mathcal{H}_4^{\mathbb{C}} \arrow[r, "\oplus", phantom]                                        & \mathcal{H}_4^{\mathbb{C}} \arrow[r, "\oplus", phantom]                                        & \mathcal{H}_4^{\mathbb{C}}                                 &                                     &                               \\
{}                                                       &                                                                                                &                                                                                                &                                                            &                                     &                              
\end{tikzcd}$$
\indent We now seek a convenient basis of $I_\ell$.  As we now explain, such a basis will be provided by applying the $\partial$ operators to a $\Delta$ eigenbasis of the $\mathcal{H}_\ell^\C$-component of $L^2(\underline{\C}_2)$.  For this we require the following proposition.

\begin{prop} \label{cor:evectshifting} ${}$
\begin{enumerate}[(a)]
\item We have the identity $\overline{\partial} \partial A = \tfrac{1}{4} \left( \Delta + r K \right) A$ for $A \in \Gamma(\uC_r)$.
\item    If $A \in \Gamma(\uC_r)$ is an eigentensor for $\Delta$ with eigenvalue $\lambda,$ then $\partial A \in \Gamma(\uC_{r+1})$ is an eigentensor for $\Delta$ with eigenvalue $\lambda + (2r + 1) K$ and $\overline{\partial} A \in \Gamma(\uC_{r-1})$ is an eigentensor for $\Delta$ with eigenvalue $\lambda - (2 r - 1) K$.
\item The eigenvalues of $\Delta$ on $\Gamma(\uC_r)$ are $\left(-\ell(\ell+1) + r^2 \right) K$ for $\ell = r, r+1, \ldots$ with multiplicity $2\left(2 \ell + 1 \right)$, and the corresponding eigenspaces are equal to the $\mathcal{H}^{\C}_\ell$ component of $\Gamma(\uC_r).$
\end{enumerate}
\end{prop}

\begin{proof} \indent (a) This follows immediately from Proposition \ref{prop:weitzenbock} and Proposition \ref{prop:CrLaplacian}(a).
 \\
 
 (b) From the commutation relation $[\Delta, \partial] A = (2r+1) K\, \partial A$ of Proposition \ref{prop:CrLaplacian}(b), we have
    \begin{equation*}
        \Delta ( \partial A ) = \partial ( \Delta A) + (2r+1) K \partial A = \left( \lambda + (2 r + 1) K \right) \partial A,
    \end{equation*}
    and similarly for $\Delta(\overline{\partial} A)$. \\

\indent (c) This is proven by induction on $r.$ The case $r= 0$ is equivalent to the standard fact that the spectrum of Laplacian on real-valued functions on $S^2$ is $-\ell(\ell+1)$ with multiplicity $2 \ell+1.$ Now assume the inductive hypothesis: the eigenvalues of $\Delta$ on $\Gamma(\uC_k)$ are $\left(-\ell(\ell+1) + k^2 \right) K$ for $\ell = k, k+1, \ldots,$ with eigenspaces the $\mathcal{H}^{\C}_\ell$ components of $\Gamma(\uC_k).$ Part (a) implies that $\partial : \Gamma(\underline{\mathbb{C}}_k) \to \Gamma(\underline{\mathbb{C}}_{k+1})$ is surjective, since if $A \in \Gamma(\underline{\mathbb{C}}_k)$ is an eigentensor with eigenvalue $\left(-\ell(\ell+1) + k^2 \right) K$ then $\overline{\partial} \partial A$ is a non-zero multiple of $A$ unless $\ell = k,$ but the decomposition of $\Gamma(\uC_{k+1})$ has no $\mathcal{H}_{k}^\C$ component. Part (b) then implies that the eigentensors for $\Delta$ on $\Gamma(\uC_{k+1})$ are of the form $\partial A,$ where $A$ is an eigentensor for $\Delta$ on $\Gamma(\uC_k).$ If $A \in \Gamma(\uC_k)$ has eigenvalue $\left(-\ell(\ell+1) + k^2 \right)K,$ then part (b) implies $\partial A$ has eigenvalue $\left(-\ell(\ell+1) + k^2 + 2k+1 \right)K,$ which is $\left(-\ell(\ell+1) + (k+1)^2 \right)K.$  
\end{proof}

\indent By Proposition \ref{cor:evectshifting}, we observe that if $A$ is a $\Delta$ eigentensor in the $\mathcal{H}_\ell^\C$ component of $L^2(\uC_2)$, then $\mathscr{J}$ preserves the $m(\ell)$-dimensional subspace
    \begin{equation} \label{eq:SpecialSubspace}
        W_A := \operatorname{span}(A, \partial A, \partial^2 A, \ldots, \partial^{m(\ell)-1} A) < I_\ell.
    \end{equation}
In fact, the $m(\ell)d(\ell)$-dimensional space $I_\ell = \left(\Hc_\ell^\C\right) ^{\oplus\, m(\ell)}$ is the direct sum of $d(\ell) = 2(2 \ell + 1)$ subspaces of the form $W_A$, where the sum is taken over a $\Delta$ eigenbasis of the $\Hc_\ell^\C$ component of $L^2(\uC_2)$.  So, the action of $\mathscr{J}$ on the $m(\ell)$-dimensional subspace $W_A < I_\ell$ determines its action on all of $I_\ell$.  Using this principle, we arrive at the following:

\begin{prop}\label{prop:isospectrum} Let $\ell \geq 2$ be an integer.  The Jacobi spectrum of $\Sigma_n$ restricted to the isotypical component $I_\ell = (\mathcal{H}_\ell^\C)^{\oplus m(\ell)}$ is equal to
    \begin{equation*}
        \mathrm{Spec}_{\ell} ( \Sigma_n) = \begin{cases}
            2(2\ell + 1) \cdot \mathrm{Spec}(M_{\ell}^{(\ell)}) & 2 \leq \ell \leq n-1 \\
            2(2\ell + 1) \cdot \mathrm{Spec}(M_{\ell}) & \ell \geq n,
        \end{cases}
    \end{equation*}
where $M_\ell$ is the $(n-1) \times (n-1)$ tridiagonal matrix given by
\begin{equation*}
    M_{\ell} = \frac{2}{n(n+1)} \begin{bmatrix}
    x_1 & p_{1} & 0   & \dots  &   & 0   \\
    q_1  & x_2       & p_2   & \dots  &   & 0   \\
    0  &     q_2    & \ddots         & \ddots  &  & \vdots    \\
    & & \ddots & & & \\
      &         &      &     &     x_{n-2}      & p_{n-2}   \\
      &         &      &     &     q_{n-2}      & x_{n-1}
  \end{bmatrix},
  \end{equation*}
in which
\begin{equation}
    \begin{aligned}
        x_1 &= \ell(\ell+1) - n(n+1) - 6, \\
         x_r & = \ell(\ell+1) - 2(r+1)^2, & \text{for} \:\:\: 2 \leq r \leq n-1, \\
        p_r &= -\sqrt{\frac{(n+2+r)(n-1-r)}{2n(n+1)}}(\ell+2+r)(\ell - 1 - r), & \text{for} \:\:\: 1 \leq r \leq n-2, \\
        q_r &= -\sqrt{2n(n+1)(n+2+r)(n-1-r)}, & \text{for} \:\:\: 1 \leq r \leq n-2,
    \end{aligned}
\end{equation}
and where $M_{\ell}^{(s)}$ is the principal $(s-1) \times (s-1)$ submatrix of $M_{\ell}$.  The notation $2(2\ell + 1) \cdot \mathrm{Spec}(M_\ell)$ means that each eigenvalue in the spectrum of $M_\ell$ is counted $2(2\ell + 1)$ times.
\end{prop}
\begin{proof} The result follows from expressing the action of $\mathscr{J}$ on the subspaces $W_A$ as a matrix with respect to the basis $A, \partial A, \partial^2 A, \ldots, \partial^{m(\ell)-1} A$.  Indeed, let $\mathbf{e}_r = \partial^{r-1}A$ for $1 \leq r \leq m(\ell)$.  For $2 \leq r \leq m(\ell)-1$,  by using formula (\ref{eq:Jacobi-operator}) for $\mathscr{J}$ together with Proposition \ref{cor:evectshifting}(a) and the expressions for $K$ and $c_j$, we calculate
    \begin{equation*}
        \begin{aligned}
            \mathscr{J} \mathbf{e}_r &= -\Delta(\partial^{r-1} A) - 4 \partial (c_{r+2} \partial^{r-1} A) + 4 \overline{\partial} (c_{r+1} \partial^{r-1} A ) + 2 (c_{r+1}^2 + c_{r+2}^2 - 1) \partial^{r-1} A \\
            &= \left( \ell (\ell+1) - (r+1)^2 \right) K \mathbf{e}_r - 4 c_{r+2} \mathbf{e}_{r+1} + c_{r+1} (r^2 - \ell (\ell+1))K \mathbf{e}_{r-1} + 2 (c_{r+1}^2 + c_{r+2}^2 - 1) \mathbf{e}_r \\
            &= p_{r-1} \mathbf{e}_{r-1} + x_r \mathbf{e}_r + q_r \mathbf{e}_{r+1}.
        \end{aligned}
    \end{equation*}
    The expressions for $\mathscr{J} \mathbf{e}_1$ and $\mathscr{J} \mathbf{e}_{m(\ell)}$ are computed similarly.
\end{proof}

\subsection{The index and nullity}

\indent \indent By Proposition \ref{prop:isospectrum}, the spectra of the matrices $M_\ell$ determine the Jacobi spectrum of $\Sigma_n$.  Conjugating by a diagonal matrix to clear radicals, we have that the Jacobi eigenvalues of $\Sigma_n$ are the union of the eigenvalues of the $(n-1) \times (n-1)$ tridiagonal matrices $N_{\ell}$ (for integers $\ell \geq 2$) given by
\begin{equation*}
    N_\ell = \frac{2}{n(n+1)} \begin{bmatrix}
    x_1 & y_{1} & 0   & \dots  &   & 0   \\
    z_1  & x_2       & y_2   & \dots  &   & 0   \\
    0  &     z_2    & \ddots         & \ddots  &  & \vdots    \\
    & & \ddots & & & \\
      &         &      &     &     x_{n-2}      & y_{n-2}   \\
      &         &      &     &     z_{n-2}      & x_{n-1}
  \end{bmatrix},
\end{equation*}
where
\begin{equation}\label{eq:xyzformulas}
    \begin{aligned}
        x_1 &= \ell(\ell+1) - n(n+1) - 6, & x_r &= \ell(\ell+1) - 2(r+1)^2, & \text{for} \:\:\: 2 \leq r \leq n-1, \\
        y_r &= -(\ell+2+r)(\ell - 1 - r), & z_r &= -(n+2+r)(n-1-r), & \text{for} \:\:\: 1 \leq r \leq n-2.
    \end{aligned}
\end{equation}
\indent Before continuing, we remark that $\Sigma_n$ admits certain well-known Jacobi fields.  For example, ambient $\SO(2n+1)$-rotations give rise to a space of Jacobi fields isomorphic to $\mathfrak{so}(2n+1)/\mathfrak{so}(3) \cong \bigoplus_{\ell=3}^{2n-1} \mathcal{H}_\ell$.  More generally, twistor deformations, by which we mean the $\SO(2n+1; \C)$-action on the twistor lift of $\Sigma_n$ to the twistor space $Z = \SO(2n+1)/\U(n)$, also give rise to Jacobi fields.  The space of those deformations is isomorphic to $\mathfrak{so}(2n+1; \C)/\mathfrak{so}(3;\C) \cong \bigoplus_{\ell=3}^{2n-1} \mathcal{H}_\ell^\C$. \\
\indent We proceed to analyze the spectral properties of $N_\ell$.

\begin{prop}
    We have
    \begin{equation}\label{eq:detform}
        \det (N_\ell) = \left( \frac{2}{n(n+1)} \right)^{n-1} \,\prod_{k=1}^{n-1} (\ell - 2k - 1) (\ell + 2k + 2).
    \end{equation}
\end{prop}
\begin{proof}
    {Consider $p(\ell) = \det(N_\ell)$.  The Jacobi fields arising from ambient $\SO(2n+1)$-rotations imply that $\det(N_\ell)$ must vanish at odd integers $\ell = 3, 5, \ldots, 2n-1$.  Next, noting that $y_r = -\ell(\ell+1) + (r^2 + 3r + 2)$, we see that the $\ell$ dependence in $N_\ell$ is a function of $\ell(\ell+1)$, so that $p(-\frac{1}{2} + \ell) = p(-\frac{1}{2} - \ell)$.  Due to this symmetry, the vanishing at $\ell = 3, 5, \ldots, 2n-1$ implies vanishing at $\ell = -4, -6, \ldots, -2n$.  So, the function $p(\ell) = \det (N_\ell)$ is a polynomial of degree $2(n-1)$ with leading coefficient $(\frac{2}{n(n+1)})^{n-1}$ and there is only one such polynomial vanishing at $\ell = 3, 5, \ldots, 2n-1$ and $\ell = -4, -6, \ldots, -2n$.}
\end{proof}

\begin{lem}   \label{lem:simpevals} Let $\ell \geq 2$ be an integer.
\begin{enumerate}[(a)]
\item For $\ell \geq n$, the eigenvalues of $N_\ell$ are simple.
\item For $2 \leq \ell \leq n-1$, the matrix $N_\ell$ has at least $n - \ell$ negative eigenvalues.  Furthermore, any non-negative eigenvalue of $N_\ell$ is simple.
\item The matrix $N_2$ has all eigenvalues negative.
\end{enumerate}
\end{lem}

\begin{proof}
    (a) Note that $z_r < 0$ for $1 \leq r \leq n-2$.  Moreover, when $\ell \geq n$, we also have that $y_r < 0$ for $1 \leq r \leq n-2$.  A result in the theory of tridiagonal matrices \cite{ParlettBook} states that if the off-diagonal entries of a tridiagonal matrix have the same sign and are non-vanishing, then all of its eigenvalues are simple.     \\
    
     (b) In this case, $y_{\ell-1} = 0$ so the matrix $N_\ell$ is block lower-triangular with the blocks on the diagonal of size $(\ell -1) \times (\ell - 1)$ and $(n-\ell) \times (n - \ell)$. We will consider these two blocks separately. \\
     
	\indent In the $(\ell-1) \times (\ell -1)$ block, the off-diagonal entries are $y_1, \ldots, y_{\ell - 2}$ and $z_1, \ldots, z_{\ell - 2}$, all of which are negative.   Therefore, again by \cite{ParlettBook}, all the eigenvalues of this $(\ell - 1) \times (\ell - 1)$ block are simple.  \\
     
     \indent In the $(n-\ell) \times (n - \ell)$ block, the signs are arranged as
    \begin{equation*}
    \begin{bmatrix}
    - & + & 0   & \dots  &   & 0   \\
    - & -       & +   & \dots  &   & 0   \\
    0  &     -    & \ddots         & \ddots  &  & \vdots    \\
    & & \ddots & & & \\
      &         &      &     &     -      & +   \\
      &         &      &     &     -     & -
  \end{bmatrix}.
\end{equation*}
We claim that  if a matrix has this sign pattern, then all of its eigenvalues are negative.  To verify this, set $m = n-\ell$, let $A$ be an $m \times m$ matrix with the above sign pattern, and let $a_{ij}$ be the absolute value of the $(i,j)$th entry of $A$.  Let $v = (v_1, \ldots, v_m)$ be a non-zero eigenvector of $A$ with eigenvalue $\lambda$.  Then the eigenvalue condition $Av = \lambda v$ may be rewritten as the following system:
    \begin{align}
    a_{12}v_2 & = \left( a_{11} + \lambda \right) v_1 \label{eq:eigen1} \\
    a_{k,k+1}v_{k+1} & = a_{k,k-1}v_{k-1} + \left(a_{kk} + \lambda \right) v_k, \ \ \ (2 \leq k \leq m-1) \label{eq:eigen2} \\
    0 & = a_{m,m-1} v_{m-1} + \left( a_{mm} + \lambda \right) v_m. \label{eq:eigen3}
\end{align}
Suppose for the sake of contradiction that $\lambda \geq 0$.  If $v_1 > 0$, then (\ref{eq:eigen1}) and (\ref{eq:eigen2}) imply that $v_k > 0$ for each $k \geq 2$, which violates (\ref{eq:eigen3}).  Similarly, if $v_1 < 0$, then (\ref{eq:eigen1}) and (\ref{eq:eigen2}) imply that $v_k < 0$ for each $k \geq 2$, violating (\ref{eq:eigen3}).  Thus, we must have $v_1 = 0$.  But then (\ref{eq:eigen1}) and (\ref{eq:eigen2}) imply that $v_k = 0$ for each $k \geq 2$, contradicting the assumption that $v$ is non-zero.  \\
    
   (c) As in the proof of part (b), we regard $N_2$ as block lower-triangular with blocks on the diagonal of size $1 \times 1$ and $(n-2) \times (n-2)$.  The proof of part (b) showed that the $(n-2) \times (n-2)$ block has all of its eigenvalues negative.  On the other hand, since $x_1 = -n(n+1) - 12 < 0$, we see that the $1 \times 1$ block also contributes a negative eigenvalue.  
\end{proof}

\begin{prop}\label{prop:negevalcount}
    The number of negative eigenvalues of the matrix $N_\ell$ is
    \begin{equation*}
        \operatorname{Ind}(N_\ell) = \operatorname{max} \left( \left\lfloor \frac{2n - \ell}{2} \right\rfloor , 0 \right)\!.
    \end{equation*}
\end{prop}

\begin{proof}
    Consider the eigenvalues of $N_\ell$ as depending on the real parameter $\ell \in [2,\infty)$. The determinant formula (\ref{eq:detform}) shows that $0$ is an eigenvalue of $N_\ell$ if and only if $\ell = 3, 5, \ldots, 2n-1,$ and Lemma \ref{lem:simpevals} implies that $0$ is a simple eigenvalue for these values of $\ell$. Therefore, the number of negative eigenvalues can only change as $\ell$ crosses one of these values. For $\ell \gg 0,$ all eigenvalues of $N_\ell$ are positive (because otherwise $\mathscr{J}$ would have infinitely many negative eigenvalues), and consequently the determinant formula implies that $N_\ell$ has no negative eigenvalues for $\ell > 2n-1.$ 
    
    Thinking of $\ell$ as decreasing from a value larger than $2n-1,$ the number of negative eigenvalues can only change as $\ell$ passes the values $\ell = 3, 5, \ldots 2n-1,$ and at these values the number may only increase or decrease by $1$. Lemma \ref{lem:simpevals}(c) implies that at each of these $n-1$ points, the number of negative eigenvalues must in fact increase by $1$, because this is the only way for $N_{2}$ to have $n-1$ negative eigenvalues.
\end{proof}

\indent Before calculating the index of $\Sigma_n$, we make one more remark about $N_\ell$ for $2 \leq \ell \leq n-1$.  On one hand, Proposition \ref{prop:negevalcount} implies that $N_\ell$ has $\left\lfloor \frac{2n-\ell}{2} \right\rfloor$ negative eigenvalues.  On the other hand, recalling the proof of Lemma \ref{lem:simpevals}(b), we may view $N_\ell$ as block lower-triangular, and the lower $(n-\ell) \times (n-\ell)$ block contributes $n - \ell$ negative eigenvalues.  Therefore, the index of the principal $(\ell - 1) \times (\ell - 1)$ submatrix of $N_\ell$ is  
\begin{equation}\label{eq:submatrixindex}
    \operatorname{Ind}(N_\ell^{(\ell)}) = \left\lfloor \frac{2n-\ell}{2} \right\rfloor - (n - \ell).
\end{equation}

\begin{thm} \label{thm:IndexResult}
    The index of $\Sigma_n$ is given by
    \begin{equation*}
        \operatorname{Ind}(\Sigma_n) = n(n-1)(2n+1).
    \end{equation*}
\end{thm}

\begin{proof}
    By Proposition \ref{prop:isospectrum}, we have
    \begin{equation*}
    \begin{aligned}
        \operatorname{Ind}(\Sigma_n) &= \sum_{\ell = 2}^{n-1} 2 ( 2 \ell + 1) \operatorname{Ind}(N_\ell^{(\ell)}) + \sum_{\ell = n}^{2n - 1} 2 (2\ell + 1) \operatorname{Ind}(N_{\ell}).
            \end{aligned}
    \end{equation*}
So, by Proposition \ref{prop:negevalcount} and (\ref{eq:submatrixindex}), we may calculate
        \begin{equation*}
    \begin{aligned}
    \operatorname{Ind}(\Sigma_n) & = \sum_{\ell = 2}^{n-1} 2 ( 2 \ell + 1) \left( \left\lfloor \frac{2n-\ell}{2} \right\rfloor - n + \ell \right) + \sum_{\ell = n}^{2n - 1} 2 (2\ell + 1)  \left\lfloor \frac{2n-\ell}{2} \right\rfloor  \\
        & = \sum_{\ell = 2}^{2n-1} 2 ( 2 \ell + 1)  \left\lfloor \frac{2n-\ell}{2} \right\rfloor  + \sum_{\ell = 2}^{n - 1} 2 (2\ell + 1) (-n+\ell) \\
       & = \sum_{j = 1}^{n-1} 2 ( 4j + 1)(n-j)  + \sum_{j = 1}^{n-1} 2 (4j + 3)(n-j-1)  + \sum_{\ell = 2}^{n - 1} 2 (2\ell + 1) (-n+\ell) \\
        & = n(n-1)(2n+1).
    \end{aligned}
    \end{equation*}
\end{proof}

\begin{rmk} In particular, abbreviating $[k_2, k_3, \ldots, k_m] := (\Hc_2^{\C})^{\oplus k_2} \oplus (\Hc_3^{\C})^{\oplus k_3}  \oplus \cdots \oplus (\Hc_m^\C)^{\oplus k_m}$, we have:
\begin{align*}
\mathrm{Ind}(\Sigma_2) & = \dim( [1]) = 10 \\
\mathrm{Ind}(\Sigma_3) & = \dim( [1, 1, 1] ) = 42 \\
\mathrm{Ind}(\Sigma_4) & = \dim( [1,1, 2, 1,1]) = 108 \\
\mathrm{Ind}(\Sigma_5) & = \dim( [1,1,2,2,2,1,1]) = 220 \\
\mathrm{Ind}(\Sigma_6) & = \dim( [1,1,2,2,3,2,2,1,1] ) = 390 \\
\mathrm{Ind}(\Sigma_7) & = \dim( [1,1,2,2,3,3,3,2,2,1,1] ) = 640.
\end{align*}
Numerical evidence suggests that the Jacobi eigenvalue $-2$ has an eigenspace isomorphic to $\bigoplus_{\ell = 2}^n \mathcal{H}_\ell^\C$, and therefore has multiplicity $m_1 = \sum_{\ell = 2}^n 2(2\ell + 1) = 2(n-1)(n+3)$, but we have not verified this. 
\end{rmk}

\begin{thm} \label{thm:NullityResult}
    The nullity of $\Sigma_n$ is given by
    \begin{equation*}
        \operatorname{Nul}(\Sigma_n) = \left(2n+3 \right) \left(2n-2 \right)\!.
    \end{equation*}
    In particular, every Jacobi field arises from a twistor deformation of $\Sigma_n.$
\end{thm}

\begin{proof}
The determinant formula (\ref{eq:detform}) shows that $0$ is an eigenvalue of $N_\ell$ if and only if $\ell = 3, 5, \ldots, 2n-1,$ and Lemma \ref{lem:simpevals} implies that $0$ is a simple eigenvalue for these values of $\ell$.  By Proposition \ref{prop:isospectrum}, each simple eigenvalue of $N_\ell$ contributes $2(2\ell + 1) = 4\ell + 2$ eigenvectors of the Jacobi operator.  Therefore, parametrizing $\ell = 3, 5, \ldots, 2n-1$ as $\{2j+1 \colon 1 \leq j \leq n-1\}$, we find that 
    \begin{equation*}
        \begin{aligned}
             \operatorname{Nul}(\Sigma_n) &=  \sum_{j=1}^{n-1}  4(2j+1) + 2 = (2n+3)(2n-2).
        \end{aligned}
    \end{equation*}
Finally, since $\operatorname{Nul}(\Sigma_n) = (2n+3)(2n-2) = \dim \SO(2n+1, \C) - \dim \SO(3, \C)$, it follows that every Jacobi field arises from a twistor deformation of $\Sigma_n$.
\end{proof}

\section{The high codimension case}  \label{sec:HighCodim}

\indent \indent In this short section, we prove Corollary \ref{cor:HighCodim}. Let $\Sigma^2$ be an immersed minimal surface in $S^{2n+a}$ that lies in a totally geodesic $S^{2n} \subset S^{2n+a}$.  Let $N\Sigma$ and $\widetilde{N}\Sigma$ denote the normal bundles of $\Sigma$ viewed as a submanifold of $S^{2n+a}$ and of $S^{2n}$, respectively.  Note that since the normal bundle $NS^{2n} \to S^{2n}$ is trivial \cite{Simons68}, there is a bundle isomorphism $NS^{2n} \cong \underline{\R}^a$, thereby yielding a splitting
$$N\Sigma \cong \widetilde{N}\Sigma \oplus \underline{\R}^a.$$
In particular, every section $\eta \in \Gamma(N\Sigma)$ can be expressed as $\eta = \nu + F$, where $\nu \in \Gamma(\widetilde{N}\Sigma)$ and $F \in \Gamma(\underline{\R}^a) \cong C^\infty(\Sigma; \R^a)$. \\
\indent The fact that $S^{2n} \subset S^{2n+1}$ is totally geodesic has several consequences.  First, the second fundamental form $\mathrm{I\!I}$ of $\Sigma \subset S^{2n+a}$ is equal to the second fundamental form $\widetilde{\mathrm{I\!I}}$ of $\Sigma \subset S^{2n+a}$.  Second, the normal connection $\nabla^\perp$ on $N\Sigma$ decomposes as a sum $\nabla^\perp = \widetilde{\nabla} + D$, where $\widetilde{\nabla}$ is the normal connection on $\widetilde{N}\Sigma$, and $D$ is the trivial connection on $\underline{\R}^a$.  This implies that
$$\Delta^\perp = \widetilde{\Delta} + \Delta_\Sigma,$$
where here $\Delta^\perp$, $\widetilde{\Delta}$, and $\Delta_\Sigma$ are the respective connection Laplacians on $N\Sigma$, $\widetilde{N}\Sigma$, and $\underline{\R}^a$.   \\
\indent Third, we may view $\Sigma$ as a minimal surface of both $S^{2n+a}$ and of $S^{2n}$.  In particular, letting $\mathscr{B} \colon \Gamma(N\Sigma) \to \Gamma(N\Sigma)$ and $\widetilde{\mathscr{B}} \colon \Gamma(\widetilde{N}\Sigma) \to \Gamma(\widetilde{N}\Sigma)$ denote the operators
\begin{align*}
\mathscr{B}(\eta) & = \left\langle \mathrm{I\!I}_{ij}, \eta \right\rangle \mathrm{I\!I}_{ij} & \widetilde{\mathscr{B}}(\nu) & = \langle \widetilde{\mathrm{I\!I}}_{ij}, \eta \rangle\, \widetilde{\mathrm{I\!I}}_{ij},
\end{align*}
where we abbreviate $\mathrm{I\!I}_{ij} = \mathrm{I\!I}(e_i, e_j)$ for an orthonormal frame $(e_1, e_2)$ on $T\Sigma$, the Jacobi operators $\mathscr{J}$ and $\widetilde{\mathscr{J}}$ with respect to $S^{2n+a}$ and $S^{2n}$ are given by
\begin{align*}
\mathscr{J}(\nu + F) & = -(\Delta^\perp + \mathscr{B} + 2\,\mathrm{Id})(\nu + F) \\
\widetilde{\mathscr{J}}(\nu) & = -(\widetilde{\Delta} + \widetilde{\mathscr{B}} + 2\,\mathrm{Id})(\nu).
\end{align*}
So, since $\mathrm{I\!I}_{ij} = \widetilde{\mathrm{I\!I}}_{ij}  \in \Gamma(\widetilde{N}\Sigma)$, we see that $\mathscr{B}(\nu + F) = \left\langle \mathrm{I\!I}_{ij}, \nu + F \right\rangle \mathrm{I\!I}_{ij} = \left\langle \mathrm{I\!I}_{ij}, \nu \right\rangle \mathrm{I\!I}_{ij} = \widetilde{\mathscr{B}}(\nu)$, and consequently
\begin{align*}
\mathscr{J}(\nu + F) & = -\Delta^\perp (\nu + F) - \mathscr{B}(\nu + F) - 2(\nu + F) \\
& = -\widetilde{\Delta}\nu - \Delta_\Sigma F - \widetilde{\mathscr{B}}(\nu) - 2\nu - 2F \\
& = \widetilde{\mathscr{J}}(\nu) - (\Delta_\Sigma + 2\,\mathrm{Id})(F).
\end{align*}
In other words, we have
$$\mathscr{J}\!\begin{bmatrix} \nu \\ F \end{bmatrix} = \begin{bmatrix} \widetilde{\mathscr{J}} & 0 \\ 0 & -(\Delta_\Sigma + 2\,\mathrm{Id}) \end{bmatrix} \begin{bmatrix} \nu \\ F \end{bmatrix}\!.$$
We deduce the following:

\begin{prop} Let $\Sigma^2$ be an immersed minimal surface in $S^{2n+a}$ that lies in a totally geodesic $S^{2n} \subset S^{2n+a}$.  Then its nullity and index are given by
\begin{align*}
\mathrm{Nul}(\Sigma) & = \widetilde{\mathrm{Nul}}(\Sigma) + a\, \dim( \Ker(-\Delta_\Sigma-2)), \\
\mathrm{Ind}(\Sigma) & = \widetilde{\mathrm{Ind}}(\Sigma) + a \sum_{\lambda < 0} \dim( \Ker(-\Delta_\Sigma-2 - \lambda)),
\end{align*}
where here $\widetilde{\mathrm{Nul}}(\Sigma)$ and $\widetilde{\mathrm{Ind}}(\Sigma)$ denote its nullity and index as a minimal surface in $S^{2n}$.
\end{prop}

\indent We now apply this proposition to the case where $\Sigma$ is isometric to a $2$-sphere of constant curvature $K = \frac{2}{n(n+1)}$.  It is a standard fact that the eigenvalues of $\Delta_\Sigma$ are $-\ell(\ell+1)K = -2\frac{\ell(\ell+1)}{n(n+1)}$ with multiplicity $2\ell + 1$, for $\ell \in \Z_{\geq 0}$. \\
\indent For the nullity, we observe that $\Ker(-\Delta_{\Sigma} - 2) = \{f \in C^\infty(S^2) \colon \Delta_{\Sigma} f = -2f\}$.  The $\Delta_\Sigma$-eigenvalue $-2$ arises by taking $\ell = n$, so that $\dim(\Ker(-\Delta_{\Sigma} - 2)) = 2n+1$, and hence
\begin{equation} \label{eq:NulHigh}
\mathrm{Nul}(\Sigma) = \widetilde{\mathrm{Nul}}(\Sigma) + a(2n+1).
\end{equation}
For the index, we observe that $\Ker(-\Delta_{\Sigma} - 2 - \lambda) = \{f \in C^\infty(S^2) \colon \Delta_{\Sigma} f = (-\lambda-2)f\}$.
For $\lambda < 0$, we have $-\lambda - 2 > -2$, so we consider the $\Delta_\Sigma$-eigenvalues greater than $-2$.  Such eigenvalues correspond to $\ell = 0, 1, \ldots, n-1$.  Therefore, $\dim(\Ker(-\Delta_{\Sigma} - 2 - \lambda)) = \sum_{\ell = 0}^{n-1} (2\ell + 1) = n^2$, and hence
\begin{equation} \label{eq:IndHigh}
\mathrm{Ind}(\Sigma) = \widetilde{\mathrm{Ind}}(\Sigma) + an^2.
\end{equation}
Equations (\ref{eq:NulHigh}) and (\ref{eq:IndHigh}), together with Theorems \ref{thm:IndexResult} and \ref{thm:NullityResult}, along with Calabi's classification \cite{Calabi67} of minimal $2$-spheres in $S^{2n+a}$ with constant Gauss curvature, establish Corollary \ref{cor:HighCodim}.

\section{Stability indices of the associative cone over the Bor{\r{u}}vka sphere}  \label{sec:StabIndex}

\subsection{Background}

\indent \indent In this section, we apply our setup to the deformation theory of associative submanifolds in $\mathrm{G}_2$ geometry.  To set the stage, let $e_1, \ldots, e_7$ be the standard basis of $\R^7$, let $\langle \cdot, \cdot \rangle$ denote the euclidean inner product on $\R^7$, and define a $3$-form $\varphi \in \Omega^3(\R^7)$ via
\begin{equation} \label{eq:AssocForm}
    \varphi = e^{123}+e^{145}-e^{167}+ e^{247}-e^{256}-e^{346}-e^{357},
\end{equation}
where we abbreviate $e^{ijk} = e^i \wedge e^j \wedge e^k$.  The vector cross product $\times \colon \R^7 \times \R^7 \to \R^7$ may then be defined by the equation $\langle u \times v, w \rangle = \varphi(u,v,w)$. \\
\indent It is well-known that $\varphi$ is a calibration, and that the $\GL_7(\R)$-stabilizer of $\varphi$ is isomorphic to the exceptional Lie group $\G_2$.   An oriented $3$-dimensional submanifold $N \subset \R^7$ is called an \emph{associative $3$-fold} if it is $\varphi$-calibrated, meaning that $\varphi|_N = \vol_N$.   \\
\indent Let $J$ denote the standard (non-integrable) almost-complex structure on $S^6$.  Explicitly, at a point $p \in S^6 \subset \R^7$, we define $J_p \colon T_pS^6 \to T_pS^6$ to be the map $J_p(v) = p \times v$.  An oriented surface $\Sigma \subset S^6$ is called a \emph{$J$-holomorphic curve} if $J(T\Sigma) = T\Sigma$.  Arguably the simplest linearly full $J$-holomorphic curve is the Bor{\r{u}}vka sphere $\Sigma_3 \subset S^6$. \\
\indent As mentioned in $\S$\ref{sec:Intro}, a surface $\Sigma \subset S^6$ is a $J$-holomorphic curve if and only if its cone $C = \mathrm{C}(\Sigma) = \{r \sigma \in \R^7 \colon r > 0, \sigma \in S^6\}$ is associative.  Also as discussed in $\S$\ref{sec:Intro}, the deformation theory of an associative $3$-fold with a singularity modeled on $C$ is governed by the ``stability index" of $C$.  We now define this latter term precisely.

\begin{defn}[\cite{Bera23}] Let $\Sigma \subset S^6$ be a closed connected $J$-holomorphic curve, and let $C \subset \R^7$ denote its cone.  Let $\nabla$ denote the Levi-Civita connection on $S^6$.
\begin{itemize}
\item The \emph{characteristic $\SU(3)$-connection} is given by
$$\widehat{\nabla} X = \nabla X + \frac{1}{2}(\nabla X)JX.$$
\item The \emph{Dirac operator} of $\Sigma$ is the operator $D_\Sigma \colon \Gamma(N\Sigma) \to \Gamma(N\Sigma)$ given by
\begin{equation}
D_\Sigma = f_1 \times \widehat{\nabla}^\perp_{f_1} + f_2 \times \widehat{\nabla}^\perp_{f_2}, \label{eq:Dirac}
\end{equation}
where $\widehat{\nabla}^\perp$ is the normal part of $\widehat{\nabla}$, and $\{f_1, f_2\}$ is a local oriented orthonormal frame on $\Sigma$. For $\lambda \in \R$, we let $d_\lambda$ denote the dimension of $(\lambda+1)$-eigenspace of $JD_\Sigma$:
$$d_\lambda := \dim\!\left\{ \nu \in \Gamma(N\Sigma) \colon (JD_\Sigma)\nu = (\lambda + 1)\nu \right\}\!.$$
\item The \emph{stability index of $C$} is
$$\text{s-ind}(C) := \frac{1}{2}d_{-1} + d_1 + \sum_{\lambda \in (-1,1)} d_\lambda - 7 - \dim(Z),$$
where $Z$ is the stratum of the moduli space of closed connected $J$-holomorphic curves that contains $\Sigma$.  The \emph{upper and lower stability indices of $C$} are, respectively,
\begin{align*}
\text{s-ind}_+(C) & := \frac{1}{2}d_{-1} + d_1 + \sum_{\lambda \in (-1,1)} d_\lambda - 7 - \dim(\G_2/\mathrm{H}),  \\
\text{s-ind}_-(C) & := \frac{1}{2}d_{-1} + \sum_{\lambda \in (-1,1)} d_\lambda - 7,  \
\end{align*}
where $\mathrm{H} \leq \G_2$ is the maximal subgroup of $\G_2$ that fixes $\Sigma$.  The inequalities $d_1 \geq \dim(Z) \geq \dim(\G_2/\mathrm{H})$ imply that $\text{s-ind}_+(C) \geq \text{s-ind}(C) \geq \text{s-ind}_-(C)$.
\item We say that $C$ is \emph{rigid} if $\text{s-ind}_+(C) = \text{s-ind}_-(C)$.  This is equivalent to $d_1 = \dim(\G_2/\mathrm{H})$, which means that all the infinitesimal deformations of $\Sigma$ are induced by the $\G_2$-action on $S^6$.
\end{itemize}
\end{defn}

\indent Bera shows \cite[Theorem 1.10]{Bera23} that if $\Sigma \subset S^6$ has genus $0$ (respectively, genus $1$) and is not the totally geodesic $2$-sphere, then $\text{s-ind}_-(C) \geq 5$ (resp., $\text{s-ind}_-(C) \geq 1$).  Moreover, the Harvey-Lawson $T^2$-cone in $\R^6 \times 0 \subset \R^7$, and the union of two special Lagrangian $3$-planes in $\R^6 \times 0 \subset \R^7$ with transverse intersection at the origin both have $\text{s-ind}(C) = 1$, and these are the only special Lagrangian cones with stability index $1$.  In the next subsection, we prove Theorem \ref{thm:StabilityResult}, calculating the upper and lower stability indices for the cone over the Bor{\r{u}}vka sphere $\Sigma_3 \subset S^6$. 

\subsection{Calculating the stability indices}

\indent \indent We begin by setting up a suitable moving frame for $S^6$.  For this, we recall that the Lie group $\G_2$ acts transitively on $S^6$ with stabilizer isomorphic to $\SU(3)$, and hence $S^6 = \G_2/\SU(3)$.  As such, we may regard $\pi \colon \G_2 \to S^6$ as the bundle of $\SU(3)$-frames over $S^6$.  Write the Maurer-Cartan form $\zeta \in \Omega^1(\G_2; \mathfrak{g}_2)$ as
\begin{equation*}
\zeta =    \left[ \begin {array}{ccccccc} 0&-\alpha_{{1}}&-\alpha_{{2}}&-\beta_{
{1}}&-\beta_{{2}}&-\beta_{{3}}&-\beta_{{4}}\\ \noalign{\medskip}\alpha
_{{1}}&0&-\psi_{{1}}&-\theta_{{3}}-\tfrac{1}{2}\,\beta_{{3}}&-\theta_{{4}}-\tfrac{1}{2}
\,\beta_{{4}}&\theta_{{1}}+\tfrac{1}{2}\,\beta_{{1}}&-\theta_{{2}}+\tfrac{1}{2}\,\beta_{
{2}}\\ \noalign{\medskip}\alpha_{{2}}&\psi_{{1}}&0&\theta_{{4}}-\tfrac{1}{2}\,
\beta_{{4}}&-\theta_{{3}}+\tfrac{1}{2}\,\beta_{{3}}&-\theta_{{2}}-\tfrac{1}{2}\,\beta_{{
2}}&-\theta_{{1}}+\tfrac{1}{2}\,\beta_{{1}}\\ \noalign{\medskip}\beta_{{1}}&
\theta_{{3}}+\tfrac{1}{2}\,\beta_{{3}}&-\theta_{{4}}+\tfrac{1}{2}\,\beta_{{4}}&0&\tfrac{1}{2}\,
\psi_{{1}}+\xi_{{1}}&-\tfrac{1}{2}\,\alpha_{{1}}-\xi_{{2}}&-\tfrac{1}{2}\,\alpha_{{2}}+
\xi_{{3}}\\ \noalign{\medskip}\beta_{{2}}&\theta_{{4}}+\tfrac{1}{2}\,\beta_{{4}
}&\theta_{{3}}-\tfrac{1}{2}\,\beta_{{3}}&-\tfrac{1}{2}\,\psi_{{1}}-\xi_{{1}}&0&\tfrac{1}{2}\,
\alpha_{{2}}+\xi_{{3}}&-\tfrac{1}{2}\,\alpha_{{1}}+\xi_{{2}}
\\ \noalign{\medskip}\beta_{{3}}&-\theta_{{1}}-\tfrac{1}{2}\,\beta_{{1}}&\theta
_{{2}}+\tfrac{1}{2}\,\beta_{{2}}&\tfrac{1}{2}\,\alpha_{{1}}+\xi_{{2}}&-\tfrac{1}{2}\,\alpha_{{2}}
-\xi_{{3}}&0&-\tfrac{1}{2}\,\psi_{{1}}+\xi_{{1}}\\ \noalign{\medskip}\beta_{{4}
}&\theta_{{2}}-\tfrac{1}{2}\,\beta_{{2}}&\theta_{{1}}-\tfrac{1}{2}\,\beta_{{1}}&\tfrac{1}{2}\,
\alpha_{{2}}-\xi_{{3}}&\tfrac{1}{2}\,\alpha_{{1}}-\xi_{{2}}&\tfrac{1}{2}\,\psi_{{1}}-\xi
_{{1}}&0\end {array} \right].
\end{equation*}
Then the $1$-forms $\alpha, \beta$ are semi-basic for the projection $\pi \colon \G_2 \to S^6$, while $\theta, \psi, \xi$ are connection 1-forms. Explicitly, the characteristic $\SU(3)$-connection on $S^6$, regarded as an element of $\Omega^1(\G_2; \mathfrak{su}(3))$, is
\begin{equation} \label{eq:CharConn}
    \left[ \begin {array}{cccccc} 0&-\psi_{{1}}&-\theta_{{3}}&-\theta_{{4
}}&\theta_{{1}}&-\theta_{{2}}\\ \noalign{\medskip}\psi_{{1}}&0&\theta_
{{4}}&-\theta_{{3}}&-\theta_{{2}}&-\theta_{{1}}\\ \noalign{\medskip}
\theta_{{3}}&-\theta_{{4}}&0&\tfrac{1}{2}\,\psi_{{1}}+\xi_{{1}}&-\xi_{{2}}&\xi_
{{3}}\\ \noalign{\medskip}\theta_{{4}}&\theta_{{3}}&-\tfrac{1}{2}\,\psi_{{1}}-
\xi_{{1}}&0&\xi_{{3}}&\xi_{{2}}\\ \noalign{\medskip}-\theta_{{1}}&
\theta_{{2}}&\xi_{{2}}&-\xi_{{3}}&0&-\tfrac{1}{2}\,\psi_{{1}}+\xi_{{1}}
\\ \noalign{\medskip}\theta_{{2}}&\theta_{{1}}&-\xi_{{3}}&-\xi_{{2}} & \frac{1}{2}\,\psi_{{1}}-\xi_{{1}} & 0\end {array} \right].
\end{equation}
To adapt frames to the Bor{\r{u}}vka sphere $\Sigma_3 \to S^6$, we shall regard $\SO(3) \to \Sigma_3$ as an adapted $\SO(2)$-frame bundle.  Pulling back $\zeta \in \Omega^1(\G_2; \mathfrak{g}_2)$ to $\SO(3) \leq \G_2$ and comparing with the Maurer-Cartan form $\mu \in \Omega^1(\SO(3); \mathfrak{so}(3))$ as in (\ref{eq:so3irrmc}), we have the following equations on $\SO(3)$:
\begin{align*}
\beta & = 0 & \omega_1 & = \alpha_1 & \psi_1 & = -\rho & \theta_1 & = 0 & \xi_1 & = \tfrac{5}{2}\rho \\
& & \omega_2 & = \alpha_2 &  & & \theta_2 & = 0 & \xi_2 & = 0 \\
  & &  &  &  & & \theta_3 & = \tfrac{\sqrt{15}}{6} \omega_1 & \xi_3 & = 0 \\
  &  &  & &  & & \theta_4 & = \tfrac{\sqrt{15}}{6} \omega_2
    \end{align*}
In particular, the normal part of the characteristic $\SU(3)$-connection on $N\Sigma_3$ equals the direct sum connection on $N\Sigma_3 \cong \underline{\C}_2 \oplus \underline{\C}_3$ induced by the Levi-Civita connection $\rho$.  \\
\indent Now, from $\langle J_{e_1}(v), w \rangle = \varphi(e_1, v, w)$ and (\ref{eq:AssocForm}), we calculate that the almost complex structure $J$ on $N\Sigma_3 = \uC_2 \oplus \uC_3$ is
\begin{equation*}
    J = \begin{bmatrix}
        i & 0 \\
        0 & -i
    \end{bmatrix}\!.
\end{equation*}
Using (\ref{eq:Dirac}) and (\ref{eq:CharConn}), we also calculate that the Dirac operator $D_{\Sigma_3} \colon \Gamma(N\Sigma_3) \to \Gamma(N\Sigma_3)$ is
\begin{equation*}
    D_{\Sigma_3} \!\begin{bmatrix}
        A_2 \\
        A_3
    \end{bmatrix} = \begin{bmatrix}
        0 & -2i \overline{\partial} \\
        -2i \partial & 0
    \end{bmatrix} \begin{bmatrix}
        A_2 \\ A_3
    \end{bmatrix}\!.
\end{equation*}
To obtain the spectrum of $JD_{\Sigma_3}$, we shall exploit the $\SO(3)$-invariant decomposition (\ref{eq:Isotypical}) of $L^2(\Sigma_3, N\Sigma_3)$ into isotypical summands, namely
\begin{align*}
L^2(\Sigma_3, N\Sigma_3) \cong \bigoplus_{\ell = 2}^\infty I_\ell = \Hc_2^\C \oplus \bigoplus_{\ell = 3}^\infty (\Hc_\ell^\C)^{\oplus 2}.
\end{align*}
In particular, we see that
$$I_2 = \Hc_2^\C \ \ \ \text{ and } \ \ \ I_\ell = \Hc_\ell^\C \oplus \Hc_\ell^\C \text{ for }\ell > 2.$$
We now prove:

\begin{thm} We have $\mathrm{s}\text{-}\mathrm{ind}_+(C(\Sigma_3)) = 41$ and $\mathrm{s}\text{-}\mathrm{ind}_-(C(\Sigma_3)) = 30$.  In particular, the cone over the Bor{\r{u}}vka sphere is not rigid.
\end{thm}

\begin{proof} We shall study the Dirac operator $D_{\Sigma_3}$ on the subspaces $W_A < I_\ell$ defined in (\ref{eq:SpecialSubspace}).  For $\ell = 2$, the operator $D_{\Sigma_3}$ restricted to $W_A = \mathrm{span}(A)$ is identically zero.  For $\ell > 2$, the operator $D_{\Sigma_3}$ restricted to $W_A = \operatorname{span}(A, \partial A)$ has the following matrix representation with respect to the basis $\{A, \partial A\}$:
\begin{equation*}
    \begin{bmatrix}
        0 & -\tfrac{i}{12}(6 - \ell(\ell+1)) \\
        -2i & 0
    \end{bmatrix}\!.
\end{equation*}
Therefore, the spectrum of $J D_{\Sigma_3}$ consists of $0$ with multiplicity $\dim(\Hc_2^\C) = 10$, and the numbers $ \lambda^{\pm}_{\ell} = \pm\sqrt{\frac{(\ell+3)(\ell-2)}{6}}$ for $\ell > 2$, each of multiplicity each of multiplicity $\dim(\Hc_\ell^\C) = 2(2\ell+1)$.  In particular, the eigenvalues of $JD_{\Sigma_3}$ in the interval $[0,2]$ consist of $0$, $\lambda_3^+ = 1$, $\lambda_4^+ = \sqrt{\frac{7}{3}}$, and $\lambda_5^+ = 2$ with respective multiplicities $10$, $14$, $18$, and $22$.  Recalling that $d_\lambda$ is the dimension of the $(\lambda+1)$-eigenspace of $JD_{\Sigma_3}$, we find that
\begin{align*}
d_{-1} & = 10, & d_0 & = 14, & d_{\sqrt{\frac{7}{3}} - 1} & = 18, & d_1 & = 22.
\end{align*}
Finally, noting that $\mathrm{H} = \SO(3)$, so that $\dim(\G_2/\mathrm{H}) = 11$, the result follows.

\end{proof}

\bibliography{IndexBoruvkaRefs}

\Addresses

\end{document}